\documentclass[9pt]{amsart}
\usepackage[english]{babel}
\usepackage[latin1]{inputenc}
\usepackage{mathrsfs} 
\usepackage{amsmath}
\usepackage{bbm}
\usepackage{amssymb} 

\usepackage{url}
\usepackage{amsthm}
\theoremstyle{theorem}\newtheorem{theorem}{Theorem}
\theoremstyle{definition}\newtheorem{definition}{Definition}
\theoremstyle{theorem}\newtheorem{lemma}{Lemma}
\theoremstyle{theorem}\newtheorem{proposition}{Proposition}
\theoremstyle{theorem}\newtheorem{corollary}{Corollary}
\theoremstyle{remark}\newtheorem{remark}{Remark}
\theoremstyle{remark}\newtheorem{example}{Example}

\usepackage{lipsum} 

\usepackage[T1]{fontenc}

\usepackage[usenames]{color}

\title{Asymptotic Stability of empirical processes and related functionals}

\author{Jos\'e L. Fern\'andez}
\address{Departamento de Matem\'ticas. Universidad Aut\'onoma de Madrid}
\email{joseluis.fernandez@uam.es}

\author{Enrico Ferri}
\address{Departamento de Matem\'aticas,
 Facultad de Inform\'atica, Universidade da Coru\~na}
\email{enrico.ferri@udc.es}

\author{Carlos V\'azquez}
\address{Departamento de Matem\'aticas,
 Facultad de Inform\'atica, Universidade da Coru\~na}
\email{carlosv@udc.es} 

\date{\today}

\begin{document}

\maketitle

\begin{abstract}
Let $E$ be a space of observables in a sequence of trials $\xi_n$ and define $m_n$ to be the empirical distributions of the outcomes. We discuss the almost sure convergence of the sequence $m_n$ in terms of the $\psi$-weak topology of measures, when the sequence $\xi_n$ is assumed to be stationary. In this respect, the limit variable is naturally described as a certain canonical conditional distribution. Then, given some functional $\tau$  defined on a space of laws, the consistency of the estimators $\tau(m_n)$ is investigated. Hence, a criterion for a refined notion of robustness, that applies when considering random measures, is provided in terms of the modulus of continuity of $\tau$.
\end{abstract}

\section{Introduction}

Let $E$ be a space in which any element encodes an observable in a sequence of trials $\xi_1,\xi_2,...$, and let $E^\mathbb{N}$ be the entire space of the sequences of outcomes, endowed with a background probability measure $\mathbb{P}$. In particular, throughout this paper, we assume that the observations $\xi_1,\xi_2,...$ form a stationary sequence with respect to the measure $\mathbb{P}$.

For any $n\geq 1$, let $m_n\triangleq n^{-1}\sum_{i\leq n}\delta_{\xi_i}$ be the \emph{empirical distribution} generated by the observations. Given the entire class $\mathfrak{M}_1(E)$ of laws on $E$, endowed with some proper measurable structure, each empirical mean $m_n$ may be always understood as a random element of $\mathfrak{M}_1(E)$.

This paper is motivated by the study of the asymptotic stability of the random sequence $\tau_n\triangleq\tau(m_n)$, $n\geq 1$, when the variables $\xi_n$ encode the \emph{historical data} of a certain \emph{financial risk factor} and $\tau:\mathfrak{M}_1(E)\rightarrow T$ is a certain statistic assessing the downside risk of the related exposure.
Indeed, the study of the asymptotic behaviour of the estimators $\tau_n$ is crucial to gauge the risk properly, see Cont et al. \cite{con}, F\"ollmer and Weber \cite{foll} and Kr\"atschmer et al. \cite{kra1,kra}.

In this respect, a key aspect is whether we have \emph{consistency} of the risk estimators $\tau_n$, i.e. whether such a sequence admits a proper limit in some stochastic sense.

If the variables  $\xi_n$ are independent and with  common distribution $\mu$, Varada\-rajan theorem guarantees the $\mathbb{P}$-almost sure convergence of the empirical process $m_n$ to the law $\mu$ in the weak topology, and the consistency of the estimators $\tau_n$ is directly obtained from the continuity property of the statistic $\tau$.

Nevertheless, as  highlighted by Cont et al. \cite{con} and Kou et al. \cite{kou} some commonly used  risk functionals, more precisely the functionals associated to the entire class of \emph{law-invariant convex risk measures}, fail to be continuous with respect to the weak topology of measures. Briefly, the reason lies behind the fact that the weak topology is not sensitive enough to the tail behaviour of the distributions, which,  by the way, is the main issue in risk analysis.

An approach to overcome this lack of sensitive has been proposed in Kr\"atsch\-mer et al. \cite{kra1,kra} and Z\"ahle \cite{zh2,zh1}. Its main ingredient is to introduce a proper refinement of the  topological structure, fine enough to control  the distributions of the tails, via the subspaces $\mathfrak{M}_1^\psi(E)$ of $\mathfrak{M}_1(E)$ defined in terms of \emph{gauge} functions $\psi$ and the associated $\psi$-weak topologies.

Theorem \ref{th:varadarajan_type} of Section \ref{sec:consistency} describes the $\mathbb{P}$-almost sure convergence of the empirical process $m_n$, under the $\psi$-weak topology. In particular, the limit variable is naturally described as the random measure obtained as the conditional distribution of $\xi_1$ given the \emph{shift-invariant} $\sigma$-algebra associated to the variables $\xi_n$. We point out that in the specific case when in addition the variables $\xi_n$ form an ergodic sequence, the consistency result as described in our Corollary \ref{co:consisency} has been developed by Kr\"atschmer et al. \cite{kra}.

\

When assessing the downside risk associated to some financial exposure, besides consistency, \emph{robustness} is a desirable property
of asymptotic stability. Following Hampel,  an estimation is said to be robust if small changes of the law related to the outcomes $\xi_n$ only result in small changes of the distribution characterizing the estimators $\tau_n$. While the notion of qualitative robustness has been classically developed,  cf. \cite{cuev,hampphd,hamp,hub,miz}, by means of the weak topology of measures, Kr\"atschmer et al. \cite{kra1,kra} and Z\"ahle \cite{zh2,zh1}, provide a similar version that applies to the $\psi$-weak topology as the basic topological background.

The main goal of Section \ref{sec:robustness} is to develop a criterion for the {asymptotic stability} of the estimators $\tau_n$, {which turns out to be related to the notion of qulitative robustness}, by exploiting the consistency result discussed in Corollary \ref{co:consisency}.
With this aim in mind, we formulate there a {notion of asymptotic stability} in terms of the modulus of continuity of the statistic $\tau$. Such a formulation naturally arises when dealing with random measures, and hence in the particular case of the canonical conditional distribution defined by the variables $\xi_n$, if stationarity holds.

\

%

The paper is organized as follows. In Section \ref{sec:backgroud} we describe the topological structure of the workspace that we consider throughout the paper. In Section \ref{sec:setup} we present some useful measure theoretical results. In section \ref{sec:consistency} we propose a criterion for consistency that is exploited in Section \ref{sec:robustness} in order to assess the {asymptotic stability} of the estimators $\tau_n$.

\section{Background}\label{sec:backgroud}

Let $E$ be a Polish space and let $\mathscr{E}$ be its Borel $\sigma$-algebra.

We denote by $ \mathfrak{M}_1(E)$  the family of Borel probability measures on $E$ and by  $\mathfrak{C}_b(E)$ the Banach space  of bounded continuous  functions defined on $E$, endowed with the supremum norm.

Here and in the sequel we use the notation $$\mu f \triangleq \int_E f(x)\mu(dx),$$
wherever the measure $\mu\in\mathfrak{M}_1(E)$ and the Borel function $f$ are such that  $\int_E |f(x)|\mu(dx)<+\infty$.

The \emph{weak topology} $\sigma(\mathfrak{M}_1(E),\mathfrak{C}_b(E))$ is the coarsest topology on $\mathfrak{M}_1(E)$ that renders continuous each map $\mu\in\mathfrak{M}_1(E)\mapsto\mu f$,  when $f$ runs over  $\mathfrak{C}_b(E)$.

Since $E$ is Polish, the space $\mathfrak{M}_1(E)$ endowed with the weak   topology is metrized as a complete and separable metric space by means of the Prohorov distance
\begin{equation}\label{eq:proh_dist1}
\begin{aligned}
\pi(\mu,\nu)&\triangleq \inf \lbrace \varepsilon > 0 : \mu (B)\leq \nu(B^\varepsilon) + \varepsilon, \ \text{for any $B\in\mathscr{E}$}\rbrace,\\ & \text{for any $\mu,\nu\in\mathfrak{M}_1(E)$,}
\end{aligned}
\end{equation}
where $B^\varepsilon \triangleq \lbrace x\in E : \inf_{y\in B} d(y,x)< \varepsilon \rbrace$ stands for the $\varepsilon$-hull of $B\in\mathscr{E}$, and $d$ denotes a \textit{consistent distance} i.e. a metric on $E$ that is consistent with its topological structure.

\

\subsection{Bounded Lipschitz functions  $\mathfrak{BL}(E)$}\label{sec:weak_topology_bounded_lip}

Let $\mathfrak{BL}(E)$ denote the linear space of Lipschitz bounded functions on $E$. For a function  $f\in \mathfrak{BL}(E)$ we define
$$
 \Vert f \Vert_{\mathfrak{BL}(E)}\triangleq \Vert f \Vert_\infty + \Vert f \Vert_{\mathfrak{L}(E)}$$
where  $\Vert f\Vert_{\mathfrak{L}(E)}$ is given by $\Vert f\Vert_{\mathfrak{L}(E)}\triangleq \sup_{x\neq y} \vert f(x)-f(y)\vert/d(x,y)$, where $d$ is a consistent metric on $E$. The space $\mathfrak{BL}(E)$ endowed with the norm $\Vert \cdot \Vert_{\mathfrak{BL}(E)}$ is a Banach space, (cf. \cite{dud}, Proposition 11.2.1).

\

The  weak topology on $\mathfrak{M}_1(E)$ may  be alternatively generated by means of  the space $\mathfrak{BL}(E)$ of bounded  Lipschitz functions on $E$,  instead of the space $\mathfrak{C}_b(E)$ of bounded continuous functions: the weak   topology is also the coarsest topology which renders continuous each of the mappings $\mu\in \mathfrak{M}_1(E)\mapsto \mu f$, when $f$ runs over $\mathfrak{BL}(E)$.

We recall that
\begin{equation}\label{eq:metric_beta}
\beta(\mu,\nu)\triangleq \sup\lbrace \vert (\mu-\nu)f \vert : \Vert f \Vert_{\mathfrak{BL}(E)}\leq 1 \rbrace, \qquad \text{for any $\mu,\nu\in\mathfrak{M}_1^\psi(E)$},
\end{equation}
defines a metric on $\mathfrak{M}_1(E)$ equivalent to  the Prohorov metric (\ref{eq:proh_dist1}). Hence, given a sequence $\mu_0,\mu_1,...$ in $\mathfrak{M}_1(E)$, one has that $\mu_n\rightarrow \mu_0$ in the weak topology if and only if $\beta(\mu_n,\mu_0)\rightarrow 0$, as $n\rightarrow +\infty$, (cf. \cite{dud} Theorem 11.3.3).

\begin{remark}\label{re:equivalent_metric_tot_bounded}
Observe that the distance $\beta$ as well as the Prohorov distance $\pi$ depend on the distance $d$. Besides, if $d^\prime$ is a metric on $E$ equivalent to $d$, then, with obvious notation, the corresponding distance $\pi^\prime$ is equivalent to $\pi$ and the same for $\beta^\prime$ and $\beta$. This allows us to consider the consistent metric $d$ that turns out to be more useful for our purposes. In particular, among all the distances consistent with the topology on $E$, there is one that is \textit{totally bounded}, which will be convenient to use later on; see, e.g., Theorem 2.8.2 in \cite{dud}. On the other hand, note that the specific choice of the consistent metric $d$ does not affect the separability of $E$.
\end{remark}

\

\noindent\emph{Separability issues.} We recall that the space $\mathfrak{C}_b(E)$ endowed with the supremum norm $\Vert \cdot \Vert_\infty$ is, in general,  not separable. Likewise, the space $\mathfrak{BL}(E)$ is, in  general,  not separable for the topology induced by the norm $\Vert \cdot \Vert_{\mathfrak{BL}(E)}$.

Let us define $\mathfrak{BL}_1(E)$ to be the unit ball in $\mathfrak{BL}(E)$, i.e. the set  of functions $f\in\mathfrak{BL}(E)$ such that $\Vert f\Vert_{\mathfrak{BL}(E)}\leq 1$. Recall that the family $\mathfrak{BL}_1(E)$ of $\mathfrak{BL}(E)$ depends on the actual distance $d$ on $E$ used in the definition of the norm $\Vert \cdot \Vert_{\mathfrak{BL}(E)}$

The  following proposition will play a relevant role later on in this paper.

\begin{proposition}\label{pr:separable_unit_ball}
If the consistent distance $d$ on $E$ is totally bounded, then the unit ball $\mathfrak{BL}_1(E)$ is separable for the supremum norm $\Vert \cdot \Vert_\infty$.
\end{proposition}
For the proof of Proposition \ref{pr:separable_unit_ball} we shall use the following particular case.
\begin{lemma}\label{le:separable_unit_ball_compact}
Assume further that the space $E$ is compact, then $\mathfrak{BL}_1(E)$ is separable for the supremum norm $\Vert \cdot \Vert_\infty$.
\end{lemma}
Note that the consistent metric $d$ in Proposition \ref{pr:separable_unit_ball} might fail to be complete in general, since completeness is not a topological invariant. In the special case, when $d$ is indeed complete, then $E$ turns out to be compact (cf. \cite{dud}, Theorem 2.3.1), and Proposition \ref{pr:separable_unit_ball} boils down to Lemma \ref{le:separable_unit_ball_compact}.
\begin{proof}[Proof of Lemma \ref{le:separable_unit_ball_compact}]
The result follows directly from \emph{Ascoli-Arzel\'a theorem}, (cf. \cite{dud}, Theorem 2.4.7). Indeed, since the family $\mathfrak{BL}_1(E)$ is  uniformly bounded and equicontinuous and $E$ is compact, then the family $\mathfrak{BL}_1(E)$ is compact with respect to the $\Vert\cdot \Vert_{\infty}$-topology, and, consequently,  separable.
\end{proof}

%

The following proof of Proposition \ref{pr:separable_unit_ball} is modelled upon ideas contained in the proof of Thereom 11.4.1 in \cite{dud}.
\begin{proof}[Proof of Proposition \ref{pr:separable_unit_ball}]
Let $\bar{E}$ be the completion of $E$ with respect to the metric $d$, and define $\mathfrak{BL}_1(\bar{E})$ to be the unit ball in $\mathfrak{BL}(\bar{E})$ defined by the norm $\Vert\cdot\Vert_{\mathfrak{BL}(\bar{E})}$.

Fix $g\in\mathfrak{BL}_1(E)$ and denote by $\bar{g}$ the unique extension of $g$ (see, e.g. Proposition 11.2.3 in \cite{dud}) defined on the entire $\bar{E}$ such that $\Vert \bar{g} \Vert_{\mathfrak{BL}(\bar{E})}=\Vert g \Vert_{\mathfrak{BL}(E)}$ and hence so that $\bar{g}\in\mathfrak{BL}_1(\bar{E})$.

Note that $\bar{E}$ is compact, since $E$ is assumed to be totally bounded, (cf. \cite{dud}, Theorem 2.3.1). Lemma \ref{le:separable_unit_ball_compact} gives us a  subset $\mathfrak{N}$ of $\mathfrak{BL}_1(\bar{E})$ which is countable and dense in $\mathfrak{BL}_1(\bar{E})$ with respect to  the supremum norm $\Vert \cdot \Vert_\infty$.

Thus, given $g\in\mathfrak{BL}_1(E)$, for any $\varepsilon>0$ one can find $f\in\mathfrak{N}$ so that $\Vert \bar{g}-f\Vert_{\infty}\leq \varepsilon$. On the other hand, letting $f\vert_E$ be the restriction of $f$ to the domain $E$, one has
$\Vert g - f\vert_E \Vert_\infty \leq \Vert \bar{g} - f \Vert_{\infty}\leq \varepsilon$. Hence, the family $\mathfrak{N}\vert_E$ of the functions in $\mathfrak{N}$ restricted to $E$ provides a dense and countable subset of $\mathfrak{BL}_1(E)$ in the norm $\Vert \cdot \Vert_\infty$. 
\end{proof}

\subsection{Gauges $\psi$.} Let $\psi$ be a continuous function on $E$, satisfying $\psi\geq 1$ everywhere on  $E$. Throughout the paper, $\psi$ will play the role of \emph{gauge function}. In particular,
following Follmer and Schied, see \cite{foll3}, and also \cite{kra1,kra}, we associate to such $\psi$ the space of functions $\mathfrak{C}_\psi(E)$ given by $$\mathfrak{C}_\psi(E)\triangleq \lbrace f\in \mathfrak{C}(E): \Vert f/\psi\Vert_\infty< \infty\rbrace\, ,$$ and the space of probability measures $\mathfrak{M}_1^\psi(E)$ defined by $$\mathfrak{M}_1^\psi(E)\triangleq \{\mu \in \mathfrak{M}_1(E): \mu \psi <+\infty\}\, .$$
Observe that  $\mathfrak{C}_b(E)\subseteq \mathfrak{C}_\psi(E)$ and that $\mathfrak{M}_1^\psi(E)\subseteq \mathfrak{M}_1(E)$.

\begin{definition}[$\psi$-weak topology]
The \emph{$\psi$-weak topology} $\sigma(\mathfrak{M}_1^\psi(E),\mathfrak{C}_\psi(E))$ is the
coarsest topology on $\mathfrak{M}_1^\psi(E)$ that renders continuous the maps $\mu\in \mathfrak{M}_1^\psi(E)\mapsto\mu f$, varying $f\in\mathfrak{C}_\psi(E)$.
\end{definition}

Besides the $\psi$-weak topology, in $\mathfrak{M}_1^\psi(E)$ we need to consider also the  \emph{relative weak topology} induced on $\mathfrak{M}_1^\psi(E)$ as a subspace of $\mathfrak{M}_1(E)$, endowed with the weak topology as defined above. This relative weak topology is actually $\sigma(\mathfrak{M}_1^\psi(E),\mathfrak{C}_b(E))$, the coarsest topology so that for each $f\in \mathfrak{C}_b(E)$, the mapping $\mu\in \mathfrak{M}_1^\psi(E)\mapsto \mu f$ is continuous, see, e.g., Lemma 2.53 in \cite{ali}.
The $\psi$-weak topology is in general finer than the relative weak  topology.

\

If $\psi\equiv 1$ or simply if $\psi$ is bounded above, then $\mathfrak{C}_b(E)= \mathfrak{C}_\psi(E)$,  $\mathfrak{M}_1^\psi(E)=\mathfrak{M}_1(E)$ and the $\psi$-weak topology and the relative weak topology coincide.

\

We introduce now a distance $d_\psi$ on $\mathfrak{M}_1^\psi(E)$ by
\begin{equation}\label{eq:psi_metric}
d_\psi(\mu,\nu)\triangleq \pi(\mu,\nu)+\vert (\mu-\nu)\psi\vert, \hspace{1cm} \text{for any $\mu,\nu\in\mathfrak{M}_1^{\psi}(E)$.}
\end{equation}

\

The following Proposition combines the results of  \cite{foll3} and \cite{kra1}.
\begin{proposition}\label{prop:distance psi and polish}
$\mathfrak{M}_1^\psi(E)$ endowed with the $\psi$-weak topology is a Polish space and its topology is generated by the distance $d_\psi$.

$\mathfrak{M}_1^\psi(E)$ endowed with the relative weak topology is separable.
\end{proposition}
\begin{proof}
Corollary A.45 of \cite{foll3} gives us that $\mathfrak{M}_1^\psi(E)$ endowed with the $\psi$-weak topology is Polish, in particular, metrizable. Now, Lemma 3.4 in \cite{kra1}, gives us that a sequence converges $\psi$-weakly if and only if it converges in the distance $d_\psi$.

If $(e_j)_{j\ge 1}$ is a sequence dense in $E$, then the family of convex combinations (with rational weights) of $\delta_{e_j}$ is contained in $\mathfrak{M}_1^\psi(E)$ and it is dense in $\mathfrak{M}_1(E)$, with respect to Prohorov distance. Hence, the space $\mathfrak{M}_1^\psi(E)$ is separable when endowed with the relative weak topology.
\end{proof}

\begin{remark}\label{re:convergence_psi_weak_topology}
Note that, given a sequence $\mu_0,\mu_1,...$ in $\mathfrak{M}_1^\psi(E)$, then $\mu_n\rightarrow\mu_0$ in the $\psi$-weak topology, as $n\rightarrow +\infty$, if and only if $\mu_n\rightarrow\mu_0$ in the weak topology and $\mu_n\psi\rightarrow\mu_0 \psi$,
 as $n\rightarrow +\infty$.
\end{remark}

\

\noindent\emph{The measurable structure of $\mathfrak{M}_1^\psi(E)$.}
We denote by  $\mathscr{M}$  the Borel $\sigma$-algebra on $\mathfrak{M}_1(E)$ generated by the weak   topology and by $\mathscr{M}^\psi$  the Borel $\sigma$-algebra on $\mathfrak{M}_1^\psi(E)$ generated by the $\psi$-weak   topology.

Next, we collect some properties of $\mathscr{M}$ and $\mathscr{M}^\psi$. We recall that $\mathscr{M}$ has  the following characterization.

\begin{lemma} \label{le:standard_weak_projections}
The $\sigma$-algebra $\mathscr{M}$ is generated by the projections $\pi_B:\mu\mapsto \mu(B)$, defined for $\mu\in\mathfrak{M}_1(E)$, letting $B$ vary in $\mathscr{E}$.
\end{lemma}
\begin{proof}
See, e.g., Proposition 2.2.2. in \cite{ghosh}. 
\end{proof}

\begin{lemma}\label{le:mueasurability_generic_family}
Let $\mathfrak{H}$ be a family of functions defined on a set $H$ and taking values in a measurable space $(G,\mathscr{G})$. Let $\phi$ be a $H$-valued map defined on some set $H_0$, then $\phi^{-1}(\sigma(\mathfrak{H}))=\sigma(\mathfrak{H}\circ\phi)$ on $H_0$, where $\mathfrak{H}\circ \phi\triangleq \lbrace h\circ \phi : h\in\mathfrak{H}\rbrace$.
\end{lemma}
\begin{proof}
First of all, note that $\phi^{-1}(\sigma(\mathfrak{H}))$ is a $\sigma$-algebra on $H_0$, since the map $\phi^{-1}$ preserves all the set operations. Thus, the inclusion $\sigma(\mathfrak{H}\circ\phi)\subseteq\phi^{-1}(\sigma(\mathfrak{H}))$ is immediate, since $h\circ \phi$ is $\phi^{-1}(\sigma(\mathfrak{H}))$-measurable for any $h\in\mathfrak{H}$.

Let now $\mathscr{H}_0$ be a $\sigma$-algebra on $H_0$ with respect to which $h\circ \phi$ is $(\mathscr{H}_0,\mathscr{G})$-measurable, for any $h\in\mathfrak{H}$. Clearly $\phi^{-1}(\sigma(\mathfrak{H}))\subseteq \mathscr{H}_0$. Thus, the proof concludes by considering $\mathscr{H}_0=\sigma(\mathfrak{H}\circ\phi)$. 
\end{proof}

 The next lemma states that the relative weak topology and the $\psi$-weak  topology generate the same Borel $\sigma$-algebra on $\mathfrak{M}_1^{\psi}(E)$.
\begin{lemma}\label{le:Borel_psi_weak}
The $\sigma$-algebra $\; \mathscr{M}^\psi $ is generated by the relative weak topology $\; \; \; \sigma(\mathfrak{M}_1^{\psi}(E),\mathfrak{C}_b(E))$.
\end{lemma}
Recall that a $\sigma$-algebra is said to be $(i)$ \emph{countably generated}if it is generated by a countable family of sets and $(ii)$ \emph{countably separated} if it admits an a countable family of sets separating points.
Moreover, a measurable space is said to be \emph{standard} if it is Borel-isomorphic to a Polish space.
\begin{proof}[Proof of Lemma \ref{le:Borel_psi_weak}]
Let us define $\mathscr{B}^\psi$ to be the Borel $\sigma$-algebra associated to the relative weak topology $\sigma(\mathfrak{M}_1^\psi(E),\mathfrak{C}_b(E))$. Since the $\psi$-weak  topology is finer than the weak   topology on $\mathfrak{M}_1^{\psi}(E)$, one has $\mathscr{B}^\psi\subseteq \mathscr{M}^\psi$.

On the other hand,  Proposition \ref{prop:distance psi and polish} gives us that $\mathfrak{M}_1^\psi(E)$ is separable when endowed with the relative weak topology. As a result, the $\sigma$-algebra $\mathscr{B}^\psi$ is countably generated and countably separated. Indeed, any countable base $\mathfrak{B}^\psi$ of open sets in $\sigma(\mathfrak{M}_1^\psi(E),\mathfrak{C}_b(E))$ generates the $\sigma$-algebra $\mathscr{B}^\psi$ and separates points, (cf. \cite{bog2}, \S 6.5).

Finally, Proposition \ref{prop:distance psi and polish} again  implies that the space $(\mathfrak{M}_1^\psi(E),\mathscr{M}^\psi)$ is standard, and thus, the $\sigma$-algebra $\mathscr{M}^\psi$ coincides with $\sigma(\mathfrak{B}^\psi)$, thanks to Theorem 3.3 in \cite{mac}. 
\end{proof}

The Borel $\sigma$-algebra $\mathscr{M}^\psi$ admits a characterization in terms of the projections $\pi_B$ analogous to that of  Lemma \ref{le:standard_weak_projections} for $\mathscr{M}$. This is the content of the next proposition.

\begin{proposition}\label{pr:Borel_psi_weak}
The Borel $\sigma$-algebra $\mathscr{M}^\psi$ is generated by the projections $\pi_B: \mu \mapsto \mu(B)$, defined for $\mu\in\mathfrak{M}_1^\psi(E)$, letting $B$ vary in $\mathscr{E}$.
\end{proposition}
\begin{proof}
Let $\phi:\mathfrak{M}_1^\psi(E)\hookrightarrow \mathfrak{M}_1(E)$ be the  inclusion of $\mathfrak{M}_1^\psi(E)$ into $\mathfrak{M}_1(E)$ and define $\mathfrak{H}$ to be the family consisting of the projection maps $\pi_B:\mu\in\mathfrak{M}_1(E)\rightarrow \mu(B)$, letting $B$ vary in $\mathscr{E}$. The family $\mathfrak{H}\circ\phi\triangleq \lbrace \pi_B\circ\phi : B\in\mathscr{E}\rbrace$ consists of the projections defined on $\mathfrak{M}_1^\psi(E)$.

Let us define $\mathscr{B}^\psi$ to be the Borel $\sigma$-algebra generated by the restriction $\sigma(\mathfrak{M}_1^\psi(E),\mathfrak{C}_b(E))$ of the weak   topology to $\mathfrak{M}_1^\psi(E)$. Equality $\mathscr{B}^\psi=\phi^{-1}(\sigma(\mathfrak{H}))$ holds true, since $\sigma(\mathfrak{H})=\mathscr{M}$ due to Lemma \ref{le:standard_weak_projections} and $\mathscr{B}^\psi=\phi^{-1}(\mathscr{M})$. Hence, applying Lemma \ref{le:mueasurability_generic_family}, one deduces that $\mathscr{B}^\psi=\sigma(\mathfrak{H\circ\phi})$. The stated result now follows from Lemma \ref{le:Borel_psi_weak}. 
\end{proof}

\section{Setup} \label{sec:setup}

We now introduce our reference probability space $(\Omega, \mathscr{F}, \mathbb{P})$ for our asymptotic stability results.

We let $\Omega$ denote the  set $\Omega=E^{\mathbb{N}}$ of all the sequences $\omega=(\omega_1,\omega_2,...)$ of elements of {a Polish space} $E$. The projections $\xi_1,\xi_2,...$ are the mappings $\omega\mapsto \xi_n(\omega)\triangleq \omega_n$, for  $\omega\in\Omega$ and $n \ge 1$.

We let $\mathscr{F}$ denote the $\sigma$-algebra in $\Omega$ generated by the projections $\xi_1,\xi_2,...$. This family $\mathscr{F}$ is also the Borel $\sigma$-algebra associated to the product topology in $E^{\mathbb{N}}$; it also coincides, since $E$ is separable, with  the product $\sigma$-algebra $\mathscr{E}^{\mathbb{N}}$, (cf. \cite{par}, Theorem 1.10).

We let $\mathbb{P}$ be a  probability measure defined on $(\Omega, \mathscr{F})$, so that $(\Omega, \mathscr{F}, \mathbb{P})$ is a complete probability space.

\noindent\emph{Random measures.}
By a random measure $\chi$ on $(E,\mathscr{E})$ with support in $\mathfrak{M}_1^\psi(E)$ we understand  a probability kernel $$
(\omega,B)\in \Omega\times \mathscr{E} \mapsto \chi(\omega,B)\in [0,1]\, ,$$
such that
\begin{itemize}
\item[$(i)$] the assignment $B\in \mathscr{E}\mapsto \chi(\omega, B)$ defines a probability measure in $\mathfrak{M}_1^\psi(E)$, for each fixed $\omega \in \Omega$,
\item[$(ii)$]  the mapping $\omega\in \Omega \mapsto \chi(\omega, B)$ is $\mathscr{F}$-measurable, for each fixed $B \in \mathscr{E}$.
\end{itemize}

Besides, Proposition \ref{pr:Borel_psi_weak} allows to understand $\chi$ as a random variable on $(\Omega,\mathscr{F},\mathbb{P})$ and taking values in $(\mathfrak{M}_1^\psi(E),\mathscr{M}^\psi)$.

We shall denote by $\mathscr{L}(\chi)$  the distribution induced by $\chi$ on $(\mathfrak{M}_1^\psi(E),\mathscr{M}^\psi)$ as a pullback in the usual way:
$$
\mathscr{L}(\chi)(M)=\mathbb{P}\circ\chi^{-1}(M), \quad \mbox{for any $M \in \mathscr{M}^{\psi}$}.
$$

\noindent\emph{Empirical process.} The \emph{empirical process associated to the projections $\xi$} is the sequence $m_1,m_2,...$ of random measures defined, for each $n \ge 1$, by
\begin{equation} \label{eq:empirical_process}
m_n(\omega, B)\triangleq \frac{1}{n}\sum_{i=1}^n \delta_{\xi_i(\omega)}(B), \quad  \mbox{for any $\omega\in \Omega$ and any $B\in \mathscr{E}$.}
\end{equation}
Moreover, we say that the empirical process $m_1,m_2,...$ is directed by the variables $\xi_1,\xi_2,...$
We shall always understand each $m_n$ as a random variable defined on $(\Omega, \mathscr{F},\mathbb{P})$ and with values in $(\mathfrak{M}_1^\psi(E),\mathscr{M}^\psi)$.

\noindent\emph{Statistics and estimators.} Let $T$ be a further Polish space endowed with its Borel $\sigma$-algebra $\mathscr{T}$ and with a metric $d_T$ which induces its topological structure.

  Any $(\mathscr{M}^\psi,\mathscr{T})$-measurable functional $\tau:\mathfrak{M}_1^\psi(E)\rightarrow T$ is termed a \emph{statistic} on $\mathfrak{M}_1^\psi(E)$.

  The sequence of random variables $\tau_1,\tau_2,..$ from $(\Omega,\mathscr{F},\mathbb{P})$ into $(T,\mathscr{T})$ obtained by setting, for each $n \ge 1$,
  $$
    \tau_n\triangleq \tau(m_n),$$ is called the sequence of \emph{estimators} induced by $\tau$.

The statistic $\tau:\mathfrak{M}_1^\psi(E)\rightarrow T$ is said to be $\psi$-continuous if it is continuous with respect to the $\psi$-weak  topology on $\mathfrak{M}_1^\psi(E)$ and the topology defined on $T$.
Besides,  $\tau$ is said to be \emph{uniformly} $\psi$-continuous if for any $\varepsilon>0$ there exists $\delta(\varepsilon)>0$ such that $d_T(\tau(\mu_1),\tau(\mu_2))< \varepsilon$ if $d_\psi(\mu_1,\mu_2)<\delta(\varepsilon)$.
Note that, given a random measure $\chi$ on $(E,\mathscr{E})$ with support in $\mathfrak{M}_1^\psi(E)$ and a $\psi$-continuous statistic $\tau$ on $\mathfrak{M}_1^\psi(E)$, the composition $\tau(\chi)$ is $(\mathscr{F},\mathscr{T})$-measurable.
\begin{definition}[Strong Consistency]
Given a statistic $\tau:\mathfrak{M}_1^\psi(E)\rightarrow T$ and a random measure $\chi$ on $(E,\mathscr{E})$, we say that $\tau$, or equivalently, the sequence of estimators $(\tau_n)_n$ induced by $\tau$, is \emph{strongly consistent} for $\tau(\chi)$ if one has $\mathbb{P}$-almost surely that {$\tau_n\rightarrow\tau(\chi)$}, as $n\rightarrow +\infty$.
\end{definition}

\section{Consistency} \label{sec:consistency}

Let $(\Omega,\mathscr{F},\mathbb{P})$ be the complete probability space introduced in the previous section.

\noindent\emph{Stationarity.} We denote by $\Sigma$ the shift operator on $E^\mathbb{N}$, i.e.
$$
\Sigma(x_1,x_2,\ldots)= (x_2,x_3,\ldots)\, , \quad \mbox{for any $(x_1, x_2, \ldots)\in E^\mathbb{N}$}\, .$$

The random sequence $\xi\triangleq (\xi_1,\xi_2,...)$ in $(E,\mathscr{E})$ given by the canonical projections is said to be \emph{stationary} if one has $\mathscr{L}(\xi)=\mathscr{L}(\Sigma\xi)$.

Here and in what follows, $\mathscr{L}(\xi)\triangleq\mathbb{P}\circ \xi^{-1}$ denotes the distribution on $(E^\mathbb{N},\mathscr{E}^\mathbb{N})$ induced by the random sequence $\xi$.

\noindent\emph{The shift invariant $\sigma$-algebra.} The \emph{shift invariant $\sigma$-algebra} is defined to be the collection $\mathscr{I}$ of the Borel sets $I\in\mathscr{E}^\mathbb{N}$ such that $\Sigma^{-1}(I)=I$, {and plays a crucial role in what follows.}

Observe that if the variables $\xi_n$ are i.i.d, then the $\sigma$-algebra {$\mathscr{I}$} turns out to be $\mathbb{P}$-trivial, see, e.g., Corollary 1.6 in \cite{kal3}.

\noindent\emph{Canonical random measure.} We now introduce the canonical random measure associated to the sequence $\xi$.
\begin{lemma}\label{le:regular_version_conditional}
There exists an essentially unique regular version $\upsilon$ of the conditional distribution $\mathbb{P}[\xi_1\in \cdot \ \vert {\mathscr{I}}]$.
\end{lemma}

Recall that $\upsilon$ is by definition a ${\mathscr{I}}$-measurable random probability measure over $(E,\mathscr{E})$, i.e. {in particular} a probability kernel $(\omega,B)\mapsto \upsilon (\omega,B)$ from the probability space $(\Omega,\mathscr{F},\mathbb{P})$ to $(E,\mathscr{E})$. In other terms, $\upsilon$ is what we have termed a random measure.
We refer to \S10.4 of  \cite{bog2} for background and relevance {of} regular version of conditional distributions.
\begin{proof}[Proof of Lemma \ref{le:regular_version_conditional}]
See, e.g., Lemma 10.4.3 and Corollary 10.4.6 in \cite{bog2}. Recall that $E$ is Polish and $\mathscr{E}$ is Borel, hence countably generated. 
\end{proof}

In the remainder of this paper, we refer to $\upsilon$ as the \emph{canonical random measure} associated to $\xi$.

Note that, in the case when the projections $\xi_n$ are independent with common distribution $\mu$, so that $\mathscr{L}(\xi_n)=\mu$ for each $n \ge 1$, then $\mathbb{P}$-almost surely $\upsilon=\mu$. This is so because ${\mathscr{I}}$ is $\mathbb{P}$-trivial and then
\begin{equation}
\mathbb{P}[\xi_1\in \cdot \vert {\mathscr{I}}]=\mathbb{P}[\xi_1\in \cdot]=\mu, \hspace{1cm} \text{$\mathbb{P}$-a.s.}
\end{equation}

\

 Observe that  when $\mathscr{L}(\xi_1)\in \mathfrak{M}_1^\psi(E)$, i.e. when $\int_E \psi(x) \ \mathbb{P}\circ \xi_1^{-1}(dx)<+\infty$, we have $\mathbb{P}$-almost surely that $\upsilon\in \mathfrak{M}_1^\psi(E)$.

{In next paragraphs we address the convergence of estimators.}

When the random variables $\xi_1,\xi_2,...$ are independent and identically distributed, Varadarajan's theorem (which we record below as Proposition \ref{le:varadarajan}) asserts the convergence in the weak topology of the empirical process.
\begin{proposition}\label{le:varadarajan}
Assume that the variables $\xi_1,\xi_2,...$ are independent with common law $\mu\in\mathfrak{M}_1(E)$, then, $\mathbb{P}$-almost surely $m_n\rightarrow \mu$ in the weak topology, as $n\rightarrow\infty$.
\end{proposition}
\begin{proof}
See, e.g., Theorem 11.4.1 in \cite{dud}. 
\end{proof}

{Analogously, for the $\psi$-weak topology in $\mathfrak{M}_1^\psi(E)$ we also have the following Varadarajan type theorem, which will be proved later by using Propositions \ref{pr:UGC_ball} and \ref{pr:varadarajan_std_weak}.}

{
\begin{theorem}\label{th:varadarajan_type}
If $\xi$ is stationary and such that $\mathscr{L}(\xi_1)\in\mathfrak{M}_1^{\psi}(E)$, then $\mathbb{P}$-almost surely $m_n\rightarrow \upsilon$  in the $\psi$-weak topology, as $n\rightarrow +\infty$.
\end{theorem}
}

Recall that a family $\mathfrak{G}$ of Borel functions on $E$ is said to be an \emph{universal Glivenko-Cantelli} class if
\begin{equation}
\sup \lbrace\vert (m_n - \mu) f \vert : f\in\mathfrak{G}\rbrace \rightarrow 0, \hspace{1cm} \text{ $\mathbb{P}$-a.s. as $n\rightarrow +\infty$,} \nonumber
\end{equation}
whenever the variables $\xi_1,\xi_2,...$ that direct the empirical process $(m_n)_n$ are independent with common generic distribution $\mu \in\mathfrak{M}_1(E)$.

\
%
%
%
\begin{proposition}\label{pr:UGC_ball}
The unit ball $\mathfrak{BL}_1(E)$ constitutes an universal Glivenko-Cantelli class.
\end{proposition}
\begin{proof}
Firstly, if $\xi_1,\xi_2,...$ are independent with common generic distribution $\mu\in\mathfrak{M}_1(E)$, then, according to Lemma \ref{le:varadarajan}, $\mathbb{P}$-almost surely $m_n\rightarrow \mu$ in the weak topology as $n\rightarrow +\infty$. Therefore, as discussed in Section \ref{sec:weak_topology_bounded_lip}, we obtain that $\beta(m_n, \mu) \to 0$, as $n \to +\infty$. 
\end{proof}

\begin{proposition}\label{pr:varadarajan_std_weak}
If $\xi$ is stationary, then $\mathbb{P}$-almost surely $m_n\rightarrow \upsilon$ in the  weak topology, as $n\rightarrow +\infty$.
\end{proposition}
\begin{proof}
According to Remark \ref{re:equivalent_metric_tot_bounded}, we now use a totally bounded metric on $E$ to define the norm $\Vert\cdot \Vert_{\mathfrak{BL}_1(E)}$. The unit ball $\mathfrak{BL}_1(E)$ is an uniformly bounded family of Borel functions on $E$. Moreover, it is separable for the supremum norm, due to Proposition \ref{pr:separable_unit_ball}.

Thus, since $\mathfrak{BL}_1(E)$ forms an universal Glivenko-Cantelli class according to Proposition \ref{pr:UGC_ball}, Theorem 1.3 combined to Corollary 1.4 in \cite{van} apply, and in particular we have
\begin{equation}\label{eq:convergence_UGC_stationary}
\sup\lbrace\vert (m_n-\upsilon)f\vert : f\in\mathfrak{BL}_1(E) \rbrace \rightarrow 0, \hspace{1cm} \text{$\mathbb{P}$-a.s. as $n\rightarrow +\infty$}
\end{equation}
The proof now concludes since \eqref{eq:convergence_UGC_stationary} gives that $\beta(m_n, \upsilon)\to 0 $ almost surely, as $n \to +\infty$, and so that   $\mathbb{P}$-almost surely $m_n\rightarrow \upsilon$ in the weak topology, as $n\rightarrow +\infty$. 
\end{proof}

\begin{proof}[{Proof of Theorem \ref{th:varadarajan_type}}]
According to Proposition \ref{pr:varadarajan_std_weak}, we have that $\mathbb{P}$-almost surely $m_n\rightarrow \upsilon$ in the relative weak topology, as $n\rightarrow +\infty$. Thus, in view of Remark \ref{re:convergence_psi_weak_topology}, it remains to show that
\begin{equation}
\vert m_n \psi - \upsilon\psi \vert \rightarrow 0, \hspace{1cm} \mbox{$\mathbb{P}$-almost surely as $n\rightarrow +\infty$}. \nonumber
\end{equation}

We apply now \emph{von Neumann's version of Birkhoff's ergodic  theorem} as stated in \cite{kal}, Theorem 9.6. Using the notation therein, consider as space $S\triangleq \Omega= E^{\mathbb{N}}$, as transformation $T\triangleq \Sigma$,  the shift operator defined on it,  and as measurable function $f\triangleq \psi \circ \xi_1$.

Notice that the shift operator $T$ preserves the measure $\mathbb{P}$  since $\xi$ is assumed to be stationary. If we write $\mu$ for the common law $\mu=\mathscr{L}(\xi_1)$, then a change of variables gives that
 $$
 \int_S f d\mathbb{P}=\int_\Omega \psi(\xi_1) d\mathbb{P}=\int_E \psi(x) \,  \mathbb{P}\circ \xi_1^{-1}(dx)=\int_E \psi\, d\mu\, . $$ Since by hypothesis $\mathscr{L}(\xi_1)=\mu \in \mathfrak{M}_1^\psi(E)$, we have $
\int_E \psi\, d\mu=\mu \psi <+\infty$ and therefore  $f\in L^1(\Omega,\mathscr{F},\mathbb{P})$.

Hence, we conclude that
\begin{equation}
m_n \psi\rightarrow \mathbb{E}[ \psi(\xi_1)\vert {\mathscr{I}}], \hspace{1cm} \mbox{$\mathbb{P}$-almost surely as $n\rightarrow +\infty$.} \nonumber
\end{equation}
Finally, according to the \emph{disintegration theorem} (cf. \cite{kal}, Theorem 5.4), we may recast  the limit variable and  write
\begin{equation}
\mathbb{E}[ \psi(\xi_1)\vert {\mathscr{I}}]= \int_E \psi\, d\mathbb{P}[\xi_1\in \cdot \vert {\mathscr{I}} ]= \upsilon \psi, \hspace{1cm} \mbox{$\mathbb{P}$-almost surely.
} \nonumber
\end{equation} 
\end{proof}

\begin{remark}[\emph{Ergodicity}]
Under the same hypotheses of Theorem \ref{th:varadarajan_type}, but assuming additionally that $\xi$ is ergodic, or equivalently that its distribution $\mathscr{L}(\xi)$ is ergodic with respect to the shift operator $\Sigma$, i.e. $\mathbb{P}\lbrace\xi\in I\rbrace\in\lbrace 0,1\rbrace$ for any $I\in \mathscr{I}$, we easily get that $\mathbb{P}$-almost surely the empirical process $(m_n)_n$ converges to $\mu=\mathscr{L}(\xi_1)$ in the $\psi$-weak topology, as $n\rightarrow +\infty$, since the $\sigma$-field $\xi^{-1}\mathscr{I}$ turns out to be $\mathbb{P}$-trivial in such a case.
\end{remark}

The next result is an immediate consequence of Theorem \ref{th:varadarajan_type}.
\begin{corollary}[\emph{Strong Consistency}]\label{co:consisency}
If $\xi$ is stationary such that $\mathscr{L}(\xi_1)\in \mathfrak{M}_1^{\psi}(E)$ and $\tau:\mathfrak{M}_1^{\psi}(E)\rightarrow T$ is $\psi$-continuous, then the sequence of estimators $(\tau_n)_n$ is strongly consistent for $\tau(\upsilon)$.
\end{corollary}

{
\begin{remark}
The proofs of Theorem \ref{th:varadarajan_type} and Corollary \ref{co:consisency} are in line with the one of Theorem 2.6 in \cite{kra}, which further assume the sequence $\xi$ to be ergodic. In particular, within the cited work the role of the map $\tau$ in Corollary \ref{co:consisency} is recovered by a law-invariant convex risk measure and Theorem 6.6 in \cite{par} is considered instead of Corollary 1.4 in \cite{van}.
\end{remark}
}

Suppose that $\xi$ describes the outcome in a sequence of trials. According to the present framework, any element of the sample space $\Omega\triangleq E^\mathbb{N}$ may be understood as a
 path of $\xi$. On the other hand,
the random measure $\upsilon$ is a regular version of the distribution induced by the single variable within the process $\xi$, conditioned on the information encoded {by} the shift invariant $\sigma$-algebra $\mathscr{I}$. In this respect, each $\omega\in\Omega$ completely describes the limit distribution $\upsilon(\omega,\cdot \ )$, which may be understood as the best available description of the outcomes.

\section{{Asymptotic stability}}\label{sec:robustness}

{Let $\theta$ be a $\mathscr{E}$-measurable endomorphism over $E$, {i.e. a $(\mathscr{E},\mathscr{E})$-measurable map from $E$ to itself}, and set
\begin{equation}\label{eq:perturbation_probaility}
\lambda_{{\mathbb{P}},\theta}(\alpha)\triangleq \mathbb{P}\lbrace d_\psi (\upsilon,\upsilon\circ\theta^{-1})>\alpha\rbrace, \hspace{1cm} \text{for any $\alpha >0$,}
\end{equation}
where the random variable $\upsilon\circ\theta^{-1}$ is defined by 
\begin{equation}\label{eq:upsilon_theta}
(\omega,B)\in \Omega\times \mathscr{E} \mapsto (\upsilon\circ\theta^{-1})(\omega,B)\triangleq \upsilon(\omega,\theta^{-1}(B)).
\end{equation}
} 
Note that the function (\ref{eq:perturbation_probaility}) is well defined, since  $d_\psi$ is trivially $(\mathscr{M}^\psi\otimes\mathscr{M}^\psi)$-measurable.
Moreover, observe that $\lambda_\theta$ is a decreasing function in $\alpha>0$ and that  $\lambda_{{\mathbb{P}},\theta} (\alpha)\rightarrow 1$ as $\alpha\rightarrow 0$ and $\lambda_{{\mathbb{P}},\theta}(\alpha)\rightarrow 0$ as $\alpha\rightarrow +\infty$, via monotonicity arguments.

{Assume that the statistic $\tau$ is uniformly $\psi$-continuous and that  $\kappa$ is a modulus of continuity of $\tau$, i.e. a  continuous strictly increasing map $[0,+\infty]\rightarrow [0,+\infty]$ such that
\begin{equation}
d_\tau(\tau(\mu_1),\tau(\mu_2))\leq \kappa(d_\psi(\mu_1,\mu_2)), \hspace{1cm} \text{for any $\mu_1,\mu_2\in \mathfrak{M}_1^\psi(E)$.} \nonumber
\end{equation}
}

{
Note that, since $\lambda_{\mathbb{P},\theta}(\alpha) \leq \kappa(\alpha)$ for $\alpha$ large enough, we are allowed to set}

\begin{equation}\label{eq:norm_theta}
\Vert\theta \Vert_{\mathbb{P},\kappa}\triangleq\inf\lbrace \alpha>0 : \lambda_{{\mathbb{P}},\theta}(\alpha){\leq}\kappa(\alpha)\rbrace.
\end{equation}
\begin{example}
Assume $\xi$ to be stationary and ergodic, i.e. $\mathbb{P}\lbrace\xi\in I\rbrace\in\lbrace 0,1\rbrace$ for any $I\in \mathscr{I}$. As a consequence, the sequence $\theta(\xi_1),\theta(\xi_2),...$ turns out to be stationary and ergodic (see Lemma 9.1 combined with Lemma 9.5 in \cite{kal}) and hence 
\begin{eqnarray}
\upsilon &=& \mathbb{P}[\xi_1\in \cdot \vert \mathscr{I}] = \mathbb{P}\circ \xi_1^{-1} \hspace{.5cm} \text{a.s.} \nonumber \\
\upsilon\circ\theta^{-1} &=& \mathbb{P}[\theta(\xi_1)\in\cdot\vert \mathscr{I}] = \mathbb{P}\circ\theta(\xi_1)^{-1} \hspace{.5cm} \text{a.s.} \nonumber
\end{eqnarray}
Next, assume $\mathbb{P}\circ \xi_1^{-1}=\delta_{x_1}$ for some $x_1\in E$, and suppose that $\theta(x_1)=x_2$, which implies $\mathbb{P}\circ\theta(\xi_1)^{-1}=\delta_{x_2}$. 
Then, one has $d_\psi(\upsilon,\upsilon\circ\theta^{-1})=d_\psi(\delta_{x_1},\delta_{x_2})$ a.s. where 
\begin{equation}
d_\psi(\delta_{x_1},\delta_{x_2}) =  \min\lbrace  d_E(x_1,x_2),1 \rbrace + \vert \psi(x_1) - \psi(x_2) \vert \hspace{.5cm} \text{a.s.} \nonumber
\end{equation}
and as a direct consequence, the following bound holds
\begin{equation}
\Vert \theta \Vert_{\mathbb{P},\kappa}\leq \min\lbrace  d_E(x_1,x_2),1 \rbrace + \vert \psi(x_1) - \psi(x_2)\vert. \nonumber
\end{equation} 
\end{example}
\begin{lemma}\label{le:qualitative_robustness:theta}
If $\tau$ is uniformly $\psi$-continuous and it admits $\kappa$ as modulus of continuity, then {$\pi(\mathbb{P}\circ \tau(\upsilon)^{-1},\mathbb{P}\circ \tau(\upsilon\circ \theta^{-1})^{-1})\leq \kappa(\Vert\theta \Vert_{\mathbb{P},\kappa})$}.
\end{lemma}
\begin{proof}
Let $C\in\mathscr{T}$ and fix $\alpha > 0$ such that $\lambda_{{\mathbb{P}},\theta}(\alpha){\leq}\kappa(\alpha)$. Since $\tau$ is $\psi$-continuous, then $\tau^{-1}(C)\in\mathscr{M}^\psi$. In particular, for any $A\in\mathscr{M}^\psi$, we denote by $A^\varepsilon\triangleq \lbrace \mu \in \mathfrak{M}_1^\psi(E): d_\psi(\mu,\nu)\leq \varepsilon, \text{ for some $\nu\in A$}\rbrace$ the $\varepsilon$-hull of $A$ defined in terms of the metric $d_\psi$.

Notice that $[\tau^{-1}(C)]^\alpha \subseteq \tau^{-1}(C^{\kappa(\alpha)})$ in $\mathfrak{M}_1^{\psi}(E)$, since $\tau$ is uniformly $\psi$-continuous and admits $\kappa$ as modulus of continuity, where the $\kappa(\alpha)$-hull $C^{\kappa(\alpha)}$ of $C$ is defined in terms of the metric $d_T$. Hence, ${\upsilon\circ \theta^{-1}}\in [\tau^{-1}(C)]^\alpha$ implies ${\upsilon\circ \theta^{-1}}\in \tau^{-1}(C^{\kappa(\alpha)})$, and in particular one has that $\mathbb{P}\lbrace {\upsilon\circ \theta^{-1}} \in [\tau^{-1}(C)]^\alpha\rbrace \leq \mathbb{P}\circ \tau({\upsilon\circ \theta^{-1}})^{-1}(C^{\kappa(\alpha)})$. Thus,
\begin{eqnarray}
\mathbb{P}\circ \tau(\upsilon)^{-1}(C) &\leq & \mathbb{P}\lbrace d_\psi (\upsilon,{\upsilon\circ \theta^{-1}})>\alpha\rbrace + \mathbb{P}\lbrace {\upsilon\circ \theta^{-1}} \in [\tau^{-1}(C)]^\alpha \rbrace \nonumber \\
&\leq & \kappa(\alpha) + \mathbb{P}\circ \tau({\upsilon\circ \theta^{-1}})^{-1} \big(C^{\kappa(\alpha)}\big). \nonumber
\end{eqnarray}
Then, since the choice of $C\in \mathscr{T}$ is arbitrary, one has that
\begin{equation}
\pi\big(\mathbb{P}\circ \tau(\upsilon)^{-1},\mathbb{P}\circ \tau({\upsilon\circ \theta^{-1}})^{-1}\big)\leq \kappa(\alpha). \nonumber
\end{equation}
The proof is concluded by letting $\alpha$ tend to $\Vert\theta\Vert_{\mathbb{P},\kappa}$, while invoking the continuity of $\kappa$. 
\end{proof}
\begin{remark}
Note that, if $\kappa$ is defined to be the identity on $(0,+\infty)$, then (\ref{eq:norm_theta}) boils down to the \emph{Ky Fan distance} between $\upsilon$ and $\upsilon\circ\theta$, which are understood as random variables with values in $(\mathfrak{M}_1^\psi(E),\mathscr{M}^\psi)$, (cf. \cite{dud}, \S 9.2). In particular, when looking at Lemma \ref{le:qualitative_robustness:theta}, this is the case when $\tau$ is a contraction.
\end{remark}

{
\begin{theorem}[{Asymptotic stability}]\label{th:qualitative_robustness_theta}
Let $\xi$ be stationary such that $\mathscr{L}(\xi_1)\in \mathfrak{M}_1^\psi(E)$ and assume that $\mathscr{L}(\theta(\xi_1))\in \mathfrak{M}_1^\psi(E)$. If $\tau$ is uniformly $\psi$-continuous and it admits $\kappa$ as modulus of continuity, then
\begin{equation}\label{eq:qualitative_robustness_theta}
\limsup_{n\geq 1} \ \pi (\mathbb{P}\circ\tau(m_n)^{-1}, \mathbb{P}\circ\tau(m_n\circ\theta^{-1})^{-1})\leq \kappa(\Vert\theta\Vert_{\mathbb{P},\kappa}).
\end{equation}
\end{theorem}
}
\begin{proof}
By the triangle inequality,
{
\begin{equation}\label{eq:inequaliy_theta1}
\begin{aligned}
\pi (\mathbb{P}\circ\tau(m_n)^{-1}, & \mathbb{P}\circ\tau(m_n\circ\theta^{-1})^{-1}) \leq \\
& \pi(\mathbb{P}\circ\tau(m_n)^{-1},\mathbb{P}\circ\tau(\upsilon)^{-1}) \\&+ \pi(\mathbb{P}\circ\tau(\upsilon)^{-1}, \mathbb{P}\circ\tau(\upsilon\circ\theta^{-1})^{-1}) \\&+ \pi(
\mathbb{P}\circ\tau(\upsilon\circ\theta^{-1})^{-1},
 \mathbb{P}\circ\tau(m_n\circ\theta^{-1})^{-1} ).
\end{aligned}
\end{equation}
}
Since $\tau$ is uniformly $\psi$-continuous and it admits $\kappa$ as modulus of continuity, Lemma \ref{le:qualitative_robustness:theta} applies. Thus, we deduce from inequality (\ref{eq:inequaliy_theta1}) that
{
\begin{equation}\label{eq:inequaliy_theta2}
\begin{aligned}
\limsup_{n\geq 1}  \ & \pi (\mathbb{P}\circ\tau(m_n)^{-1}, \mathbb{P}\circ\tau(m_n\circ\theta^{-1})^{-1})\leq 
\kappa(\Vert \theta \Vert_{\mathbb{P},\kappa}) \\ 
 &+ \limsup_{n\geq 1} \ \pi(\mathbb{P}\circ\tau(m_n)^{-1},\mathbb{P}\circ\tau(\upsilon)^{-1}) \\ 
 &+  \limsup_{n\geq 1} \ \pi( \mathbb{P}\circ\tau(m_n\circ\theta^{-1})^{-1}, \mathbb{P}\circ\tau(\upsilon\circ\theta^{-1})^{-1}). \nonumber
\end{aligned}
\end{equation}
}

Since $\xi$ is assumed to be stationary such that $\mathscr{L}(\xi_1)\in \mathfrak{M}_1^\psi(E)$ and $\tau$ is $\psi$-continuous, the result described in Corollary \ref{co:consisency} guarantees that $\mathbb{P}$-almost surely  ${\tau(m_n)}\rightarrow \tau(\upsilon)$, as $n\rightarrow +\infty$, and hence 
{
\begin{equation}
\limsup_{n\geq 1} \pi(\mathbb{P}\circ\tau(m_n)^{-1},\mathbb{P}\circ\tau(\upsilon)^{-1}) = 0. \nonumber 
\end{equation}
}

On the other hand, the sequence $\theta(\xi_1),\theta(\xi_2),...$ is stationary since $\xi$ is stationary (see Lemma 9.1 in \cite{kal}). Moreover, note that
\begin{equation}
m_n\circ\theta^{-1} \triangleq \frac{1}{n}\sum_{i\leq n} \delta_{\xi_i}\circ\theta^{-1} = \frac{1}{n}\sum_{i\leq n} \delta_{\theta(\xi_i)}. \nonumber
\end{equation}
Then, since $\mathscr{L}(\theta(\xi_1))\in \mathfrak{M}_1^\psi(E)$, Corollary \ref{co:consisency} guarantees that $\mathbb{P}$-almost surely  ${\tau(m_n\circ\theta^{-1})}\rightarrow \tau(\upsilon_\theta)$, as $n\rightarrow +\infty$, where we write $\upsilon_\theta$ to denote a regular version of the conditional distribution $\mathbb{P}[\theta(\xi_1)\in \cdot\vert \mathscr{I}]$. On the other hand, notice that  $\mathbb{P}$-almost surely $\mathbb{P}[\theta(\xi_1)\in \cdot\vert \mathscr{I}]= \mathbb{P}[\xi_1 \in \theta^{-1}(\cdot)\vert \mathscr{I}]$ and hence $\upsilon_\theta = \upsilon\circ \theta^{-1}$. Thus, 
\begin{equation}
\limsup_{n\geq 1} \ \pi( \mathbb{P}\circ\tau(m_n\circ\theta^{-1})^{-1}, \mathbb{P}\circ\tau(\upsilon\circ\theta^{-1})^{-1}) = 0. \nonumber
\end{equation}

\end{proof}

\begin{remark}\label{re:intepretation_qualitative_robustness}
In the case when $\xi$ describes the outcomes in a sequence of trials, we may understand  the action of the endomorphism $\theta$ as a perturbation of the available dataset and the function $\lambda_\theta$ defined in (\ref{eq:perturbation_probaility}) measures the impact of such a perturbation in terms of the random measure $\upsilon$.


It is easy to realize that, when the perturbation procedure encoded by the action of the map $\theta$ does not change appreciably the random measure $\upsilon$ in the stochastic sense provided by (\ref{eq:perturbation_probaility}), then one should expect $\Vert \theta \Vert_{\mathbb{P},\kappa}$ to be small. In particular, this form of continuity is properly assessed in terms of $\kappa$.

{
In the particular case when $\tau$ admits $\kappa$ as modulus of continuity, Theorem \ref{th:qualitative_robustness_theta} guarantees
that any changes in the law of $\tau_n$ due to small perturbations at the level of the dataset encoded by the map $\theta$ are asymptotically gauged by the relation described in (\ref{eq:qualitative_robustness_theta}) by means of the terms $\kappa(\Vert\theta\Vert_{\mathbb{P},\kappa})$.}

{
For instance, {fix a function $f$ such that $f-\psi\in \mathfrak{BL}_1(E)$} and consider $\tau_f$ to be the functional on $\mathfrak{M}_1^\psi(E)$ defined by setting $\mu\in \mathfrak{M}_1^{\psi}(E)\mapsto\tau_f(\mu) \triangleq \mu f$. Notice that
{
\begin{equation}
\vert \tau_f(\mu)-\tau_f(\nu) \vert \leq \beta(\mu,\nu) + \vert (\mu-\nu)\psi \vert \leq C d_\psi(\mu,\nu), \hspace{.5cm}  \text{for any $\mu,\nu\in\mathfrak{M}_1^{\psi}(E)$,} \nonumber
\end{equation}
}
for some positive constant $C$, since the distance $\beta$ as given in (\ref{eq:metric_beta}) is equivalent to Prohorov metric $\pi$.  {Hence, $\tau$ is uniformly $\psi$-continuous and it admits $\kappa(\alpha)\triangleq C\alpha$, for $\alpha\geq 0$, as modulus of continuity.} According to Theorem \ref{th:qualitative_robustness_theta}, the changes in the law of the sample mean of $f$,
\begin{equation}
m_n f = \frac{1}{n}\sum_{i=1}^n f(\xi_i),
\end{equation}
with respect to the metric $\pi$, when considering the sequence $\theta(\xi_1),\theta(\xi_2),...$ are asymptotically controlled by the term
\begin{equation}
\Vert\theta \Vert_{\mathbb{P},\kappa}=\inf\lbrace \alpha>0 : \lambda_{{\mathbb{P}},\theta}(\alpha){\leq}C\alpha\rbrace. \nonumber
\end{equation}
}
\end{remark}

{
\begin{remark}[Qualitative Robustness]\label{re:qualitaive_robusntess}
Theorem \ref{th:qualitative_robustness_theta} and its interpretation in Remark \ref{re:intepretation_qualitative_robustness} are in line with the notion of qualitative robustness as discussed in Kr\"atschmer et al. \cite{kra1,kra} and Z\"ahle \cite{zh2,zh1}.
Following these cited authors, any statistic $\tau$ is said to be qualitative robust if small changes of the law related to the outcomes $\xi_n$ only result in small changes of the distribution characterizing the estimators $\tau_n$, for $n$ large enough.
\end{remark}
}

\begin{remark}[{Asymptotic Stability} and Elicitability]
\emph{Elicitability} provides a widely discussed aspect in evaluating point forecasts; for background see for instance \cite{fiss,gne,osb,zie}. In this respect, assume that
the statistic $\tau$ is \emph{elicitable}, relative to the class $\mathfrak{M}_1^\psi(E)$, by considering some \emph{strictly consistent scoring function} $S:T\times E \rightarrow [0,+\infty)$. Moreover assume that $\tau$ is uniformly continuous with respect to the functional $(\mu,\nu)\mapsto \tilde{S}(\mu,\nu)\triangleq \int_E S(\tau(\mu),x)\nu(dx)$ in the sense that
\begin{equation}\label{eq:scoring_uniform_continuity}
d_T(\tau(\mu),\tau(\nu))\leq \kappa(\tilde{S}(\mu,\nu)), \hspace{1cm} \text{for any
 $\mu,\nu\in\mathfrak{M}_1^\psi(E)$},
\end{equation}
for some non-negative continuous and increasing function $\kappa$ vanishing at zero.

Recall that $\Vert \theta \Vert_{\mathbb{P},\kappa}$ as defined in (\ref{eq:norm_theta}) implicitly depends on the metric $d_\psi$.
In a similar way, if $\tilde{S}$ is $(\mathscr{M}^\psi\otimes\mathscr{M}^\psi)$-measurable, we may define
\begin{equation}\label{eq:norm_theta_1}
\Vert \theta \Vert^{(1)}_{\mathbb{P},\kappa}\triangleq \inf\lbrace \alpha > 0 : \mathbb{P}\lbrace \tilde{S}(\upsilon,{\upsilon\circ\theta^{-1}}) > \alpha\rbrace < \kappa(\alpha) \rbrace.
\end{equation}
Hence, under condition (\ref{eq:scoring_uniform_continuity}) a similar estimate as provided in Lemma \ref{le:qualitative_robustness:theta} may be assessed in terms of (\ref{eq:norm_theta_1}), and if in addition $\tau$ is assumed to be $\psi$-continuous, $\xi$ is stationary and $\mathbb{P}$ is quasi-invariant under $\theta$, then, the arguments in the proof of Theorem \ref{th:qualitative_robustness_theta} still
remain in force and give
\begin{equation}
{\limsup_{n\geq 1} \ \pi (\mathbb{P}\circ\tau(m_n)^{-1}, \mathbb{P}\circ\tau(m_n\circ\theta^{-1})^{-1})}\leq \kappa(\Vert \theta \Vert^{(1)}_{\mathbb{P},\kappa}). \nonumber
\end{equation}

As an example, when considering $E$ and $T$ to be the real line endowed with the euclidean metric and $\psi$ the identity, if $\tau:\mu\in\mathfrak{M}_1^\psi(E)\mapsto \int_\mathbb{R} x \mu(dx)$ defines the mean and $S(x,y)\triangleq (x-y)^2$, for any $(x,y)\in\mathbb{R}^2$, then condition (\ref{eq:scoring_uniform_continuity}) is guaranteed when for instance $\kappa(z)\triangleq \sqrt{z}$, for any $z\geq 0$.

Observe also that, in the case when $\psi$ is strictly increasing and $\tau(\mu)$ is defined as the $\alpha$-quantile of the law $\mu\in\mathfrak{M}_1^\psi(E)$, for some fixed $\alpha\in(0,1)$, and the related scoring function is given by $S(x,y)\triangleq (\mathbbm{1}_{\lbrace x\geq y\rbrace}-\alpha)(\psi(x)-\psi(y))$, (see, e.g., Theorem 3.3 in \cite{gne}), condition (\ref{eq:scoring_uniform_continuity}) fails for any $\kappa$.
\end{remark}

\section{Concluding Remarks}

Theorem \ref{th:varadarajan_type}, as well as Corollary \ref{co:consisency} and Theorem \ref{th:qualitative_robustness_theta}, still remain in force when the sequence of projections $\xi_1,\xi_2,...$ displays some other forms of probabilistic symmetries.

Recall that the random sequence $\xi=(\xi_1,\xi_2,...)$ is said to be \emph{exchangeable} if and only if $\mathscr{L}(\xi_i: i\in \mathfrak{I})=\mathscr{L}(\xi_{\pi_\mathfrak{I}(i)}:i\in \mathfrak{I})$, for any finite family $\mathfrak{I}$ of indices and any permutation $\pi_\mathfrak{I}$ on it.
A numerable sequence of exchangeable random variables is always stationary, (cf. \cite{kal1}, Proposition 2.2). In particular, we get that $\mathbb{P}$-almost surely $\mathscr{I}=\sigma(\upsilon)$, (cf. \cite{kal3}, Corollary 1.6). In addition, each of the previous $\sigma$-algebras turns out to be $\mathbb{P}$-trivial in the independence setup. In this respect, we are allowed to recast the limit random variable in Theorem \ref{th:varadarajan_type} by writing $\upsilon=\mathbb{P}[\xi_1\in \cdot \ \vert \upsilon]$, where the equality shall be intended in the $\mathbb{P}$-almost surely sense. On the other hand, according to de Finetti's Theorem (cf. \cite{kal3}, Theorem 1.1), when dealing with a numerable random sequence $\xi=(\xi_1,\xi_2,...)$ in $E$, the notion of exchangeability equals a conditional form of independence, i.e. one has that $\mathbb{P}$-almost surely $\mathbb{P}[\xi\in\cdot \ \vert \upsilon]=\upsilon^{\mathbb{N}}$.

Exchangeability provides the main pillar of the Bayesian approach to the inferential analysis. More precisely, when dealing with the non parametric setup, the law induced by the random measure $\upsilon$ over the space $(\mathfrak{M}_1^{\psi}(E),\mathscr{M}^\psi)$ may be regarded as the \emph{prior distribution} of the statistical model $\xi_1,\xi_2,...\vert \upsilon \sim_{iid} \upsilon$, where the latter form of independence is to be understood in terms of de Finetti's theorem.

According to such a formulation, Theorem \ref{th:qualitative_robustness_theta} may be regarded as a form of stability obtained when the prior distribution of the model is forced to change, {by considering the random measure defined by
 identity (\ref{eq:upsilon_theta}).}




{
\section*{Acknowledgements}
The authors would like to thank an anonymous reviewer for the suggestions which clearly contributed to improve the article.
}

\bibliographystyle{acm}
\bibliography{bibfile}

\begin{thebibliography}{10}

\bibitem{ali}
{\sc Aliprantis, C.~D., and Border, K.}
\newblock {\em Infinite dimensional analysis: a hitchhiker's guide}.
\newblock Springer Science \& Business Media, 2006.

\bibitem{bog2}
{\sc Bogachev, V.~I.}
\newblock {\em Measure theory}, vol.~2.
\newblock Springer Science \& Business Media, 2007.

\bibitem{con}
{\sc Cont, R., Deguest, R., and Scandolo, G.}
\newblock Robustness and sensitivity analysis of risk measurement procedures.
\newblock {\em Quantitative Finance 10}, 6 (2010), 593--606.

\bibitem{cuev}
{\sc Cuevas, A., and Romo, J.}
\newblock On robustness properties of bootstrap approximations.
\newblock {\em Journal of Statistical Planning and Inference 37}, 2 (1993),
  181--191.

\bibitem{dud}
{\sc Dudley, R.~M.}
\newblock {\em Real analysis and probability}, vol.~74.
\newblock Cambridge University Press, 2002.

\bibitem{fiss}
{\sc Fissler, T., and Ziegel, J.~F.}
\newblock Higher order elicitability and osband's principle.
\newblock {\em The Annals of Statistics 44}, 4 (2016), 1680--1707.

\bibitem{foll3}
{\sc F{\"o}llmer, H., and Schied, A.}
\newblock {\em Stochastic finance: an introduction in discrete time}.
\newblock Walter de Gruyter, 2011.

\bibitem{foll}
{\sc F{\"o}llmer, H., and Weber, S.}
\newblock The axiomatic approach to risk measures for capital determination.
\newblock {\em Annual Review of Financial Economics 7\/} (2015), 301--337.

\bibitem{ghosh}
{\sc Ghosh, J.~K., and Ramamoorthi, R.}
\newblock {\em Bayesian Nonparametrics}.
\newblock Springer, 2003.

\bibitem{gne}
{\sc Gneiting, T.}
\newblock Making and evaluating point forecasts.
\newblock {\em Journal of the American Statistical Association 106}, 494
  (2011), 746--762.

\bibitem{hampphd}
{\sc Hampel, F.~R.}
\newblock {\em Contributions to the theory of robust estimation}.
\newblock PhD thesis, University of California, Berkeley, 1969.

\bibitem{hamp}
{\sc Hampel, F.~R.}
\newblock A general qualitative definition of robustness.
\newblock {\em The Annals of Mathematical Statistics 42}, 6 (1971), 1887--1896.

\bibitem{hub}
{\sc Huber, P.~J.}
\newblock {\em Robust statistics}.
\newblock Springer, 2011.

\bibitem{kal1}
{\sc Kallenberg, O.}
\newblock Spreading and predictable sampling in exchangeable sequences and
  processes.
\newblock {\em The Annals of Probability 16}, 2 (1988), 508--534.

\bibitem{kal}
{\sc Kallenberg, O.}
\newblock {\em Foundations of modern probability}.
\newblock Springer Science \& Business Media, 1997.

\bibitem{kal3}
{\sc Kallenberg, O.}
\newblock {\em Probabilistic symmetries and invariance principles}.
\newblock Springer Science \& Business Media, 2006.

\bibitem{kou}
{\sc Kou, S., Peng, X., and Heyde, C.~C.}
\newblock External risk measures and basel accords.
\newblock {\em Mathematics of Operations Research 38}, 3 (2013), 393--417.

\bibitem{kra1}
{\sc Kr{\"a}tschmer, V., Schied, A., and Z{\"a}hle, H.}
\newblock Qualitative and infinitesimal robustness of tail-dependent
  statistical functionals.
\newblock {\em Journal of Multivariate Analysis 103}, 1 (2012), 35--47.

\bibitem{kra}
{\sc Kr{\"a}tschmer, V., Schied, A., and Z{\"a}hle, H.}
\newblock Comparative and qualitative robustness for law-invariant risk
  measures.
\newblock {\em Finance and Stochastics 18}, 2 (2014), 271--295.

\bibitem{mac}
{\sc Mackey, G.~W.}
\newblock Borel structure in groups and their duals.
\newblock {\em Transactions of the American Mathematical Society 85}, 1 (1957),
  134--165.

\bibitem{miz}
{\sc Mizera, I.}
\newblock Qualitative robustness and weak continuity: the extreme function.
\newblock {\em Nonparametrics and Robustness in Modern Statistical Inference
  and Time Series Analysis: A Festschrift in honor of Professor Jana
  Jureckov{\'a} 1\/} (2010), 169.

\bibitem{osb}
{\sc Osband, K.}
\newblock {\em Providing incentives for better cost forecasting}.
\newblock PhD thesis, University of California, Berkeley, 1985.

\bibitem{par}
{\sc Parthasarathy, K.~R.}
\newblock {\em Probability measures on metric spaces}, vol.~352.
\newblock American Mathematical Society, 1967.

\bibitem{van}
{\sc van Handel, R.}
\newblock The universal {G}livenko--{C}antelli property.
\newblock {\em Probability Theory and Related Fields 155}, 3-4 (2013),
  911--934.

\bibitem{zh2}
{\sc Z{\"a}hle, H.}
\newblock Qualitative robustness of statistical functionals under strong
  mixing.
\newblock {\em Bernoulli 21}, 3 (2015), 1412--1434.

\bibitem{zh1}
{\sc Z{\"a}hle, H.}
\newblock A definition of qualitative robustness for general point estimators,
  and examples.
\newblock {\em Journal of Multivariate Analysis 143\/} (2016), 12--31.

\bibitem{zie}
{\sc Ziegel, J.~F.}
\newblock Coherence and elicitability.
\newblock {\em Mathematical Finance 26}, 4 (2016), 901--918.

\end{thebibliography}
\end{document}